\newif\ifarxiv
\arxivtrue

\ifarxiv
    \documentclass{article}
    \usepackage{amsthm}
    \usepackage{sidecap}
    \newtheorem{theorem}{Theorem}[section]

    \newcommand{\jkqed}[0]{}

\else

    \documentclass[smallextended]{svjour3}     
    \newcommand{\jkqed}[0]{\qed}

    \smartqed  
    \journalname{Ramanujan}
\fi

\usepackage[]{amsmath}
\usepackage[]{amsfonts}
\usepackage[]{amssymb}

\newcommand{\boxDiagramUnitlength}[0]{6mm}
\newcommand{\boxDiagramLinethickness}[0]{0.2mm}

\newcommand{\sequence}[3][1]{\ensuremath{#2_{#1}\dots #2_{#3}}}

\newcommand{\singletonSequence}[1]{%
   {\ensuremath{\langle#1\rangle}}}

\newcommand{\longsum}[3][1]{\ensuremath{#2_{#1} + \dots + #2_{#3}}}

\newcommand{\concatenate}[0]{} 

\newcommand{\sac}[0]{\ensuremath{\mathcal{A}}}
\newcommand{\nac}[0]{\ensuremath{A}}
\newcommand{\sfl}[0]{\ensuremath{L}}
\newcommand{\lexminA}[0]{\ensuremath{\mathrm{M}_\mathcal{A}}}
\newcommand{\lexsuccA}[0]{\ensuremath{\mathrm{S}_\mathcal{A}}}

\newcommand{\nump}[0]{\ensuremath{p}}

\newcommand{\cardinality}[1]{\ensuremath{|#1|}}
\newcommand{\suchthat}[0]{\ensuremath{\mid}}
\newcommand{\floor}[1]{\ensuremath{\lfloor #1 \rfloor}}

\newcommand{\stacksum}[3]{\sum_{\substack{#1\\#2}}#3}

\newcommand{\chapterNumber}[1]{ch.#1}

\newcommand{\sectionRef}[1]{Section~\ref{#1}}
\newcommand{\figureRef}[1]{Figure~\ref{#1}}

\newcommand{\equationRef}[1]{\eqref{#1}} 

\newcommand{\theoremRef}[1]{Theorem~\ref{#1}}

\begin{document}
\title{The large-parts formula for $p(n)$ }

\ifarxiv
\author{Jerome Kelleher
        \thanks{Institute of Evolutionary Biology, University of Edinburgh, King's Buildings,
    West Mains Road, EH9 3JT U.K. \texttt{jerome.kelleher@ed.ac.uk}}}

\else
\author{Jerome Kelleher}
\institute{Jerome Kelleher \at
Institute of Evolutionary Biology, 
University of Edinburgh, King's Buildings,
West Mains Road, EH9 3JT U.K.
              \email{jerome.kelleher@ed.ac.uk}           
}
\date{Received: date / Accepted: date}

\fi
\maketitle

\begin{abstract}
A new formula for the partition function $p(n)$ is developed.
We show that the number of partitions of $n$ can be expressed as the sum of 
a simple function of the two largest parts of all partitions. Specifically,
if $a_1 + \dots + a_k = n$ is a partition of $n$ with $a_1 \leq \dots \leq a_k$ 
and $a_0 = 0$, then the sum of $\lfloor(a_k + a_{k - 1}) / (a_{k-1} + 1)\rfloor$ over all
partitions of $n$ is equal to $2p(n) - 1$.
\ifarxiv
\else
\keywords{Integer Partitions \and Lexicographic Order}
\subclass{05A17 \and 11P81} 
\fi
\end{abstract}

\section{Introduction}
A sequence of positive integers $\sequence{a}{k}$ is an
ascending composition of the positive integer $n$ if $\longsum{a}{k} = n$ and $a_1 \leq \dots
\leq a_k$.
Let $\sac(n)$ be the set of all ascending compositions of $n$ for some $n
\geq 1$, and let  $\sac(n, m) \subseteq \sac(n)$ be defined for $1\leq m
\leq n$ as $\sac(n, m) = \{\sequence{a}{k}\suchthat \sequence{a}{k} \in
\sac(n)$ and $a_1 \geq m\}$. Also, let $\nac(n) = \cardinality{\sac(n)}$, 
$\nac(n, m) = \cardinality{\sac(n, m)}$ and define $a_0 = 0$ for all ascending 
compositions. We show that 
\begin{equation}\label{eqn-ak-pn}
2\nump(n)  - 1 = \stacksum{\sequence{a}{k} \in}{\sac(n)}
{\left\lfloor \frac{a_{k - 1} + a_k}{a_{k - 1} + 1}
\right\rfloor}
\end{equation}
for all $n \geq 1$ using a simple combinatorial argument.

The proof of \equationRef{eqn-ak-pn} proceeds as follows. In \sectionRef{sec-sfl}
we demonstrate that the suffix length of $\sac(n)$ with respect to 
the lexicographic ordering is $2\nump(n) - 1$.
Then, in \sectionRef{sec-lex-succ}, we derive the lexicographic succession rule for ascending 
compositions, and show that iteratively applying this rule to all elements of $\sac(n)$ 
also allows us to compute the suffix length. The main result then follows immediately 
from these observations.

\section{Suffix Length}\label{sec-sfl}
If we consider the set of ascending compositions in lexicographic order, then the 
suffix length is defined as the sum of the length of all
suffixes that differ between adjacent compositions. The suffix length is a key 
element of Kemp's general approach to the analysis of combinatorial generation 
algorithms~\cite{kemp-generating}.  
In this section we show that the suffix length of $\sac(n)$ with respect to 
the lexicographic ordering is $2\nump(n) - 1$. 
To prove this result, we require
 a recurrence to enumerate the ascending 
compositions.

To count the ascending compositions of $n$ where the first part is 
at least $m$, we first observe that there is
exactly one composition where $a_1 > \floor{n / 2}$. We then note that
$m \leq a_1 \leq \floor{n / 2}$ when there is more that one part, 
since all values 
between $m$ and $\floor{n/2}$ are legitimate smallest parts. We then get  
\begin{equation}\label{eqn-nac-iter}
\nac(n, m) = 1 + \sum_{x = m}^{\floor{n / 2}} \nac(n - x, x)
\end{equation}
for all positive integers $m\leq n$~\cite[\chapterNumber{3}]{kelleher-encoding}. 

\ifarxiv
\begin{figure}
\else
\begin{figure}\sidecaption
\fi
\setlength{\unitlength}{\boxDiagramUnitlength}
\begin{picture}(5, 7)(0,0)
\linethickness{\boxDiagramLinethickness}
\put(0,0){\line(0,1) {7}}
\put(1,0){\line(0,1) {7}}
\put(2,1){\line(0,1) {6}}
\put(3,3){\line(0,1) {4}}
\put(4,5){\line(0,1) {2}}
\put(5,6){\line(0,1) {1}}
\put(0,0){\line(1,0) {1}}
\put(0,1){\line(1,0) {2}}
\put(0,2){\line(1,0) {2}}
\put(1,3){\line(1,0) {2}}
\put(1,4){\line(1,0) {2}}
\put(2,5){\line(1,0) {2}}
\put(3,6){\line(1,0) {2}}
\put(0,7){\line(1,0) {5}}
\put(0.33,0.33){$5$}
\put(0.33,1.33){$2$}
\put(0.33,2.33){$1$}
\put(0.33,3.33){$1$}
\put(0.33,4.33){$1$}
\put(0.33,5.33){$1$}
\put(0.33,6.33){$1$}
\put(1.33,1.33){$3$}
\put(1.33,2.33){$4$}
\put(1.33,3.33){$2$}
\put(1.33,4.33){$1$}
\put(1.33,5.33){$1$}
\put(1.33,6.33){$1$}
\put(2.33,3.33){$2$}
\put(2.33,4.33){$3$}
\put(2.33,5.33){$1$}
\put(2.33,6.33){$1$}
\put(3.33,5.33){$2$}
\put(3.33,6.33){$1$}
\put(4.33,6.33){$1$}
\end{picture}
\caption{Adjacency boxes of the set $\sac(5)$ in lexicographic order. 
Theorem~\ref{the-sfl} shows that there are $2\nump(n) - 1$ of these boxes for all 
$n$. In this example, we have $p(5) = 7$ and a total of $13$ boxes.
\label{fig-box-diagram}}
\end{figure}

Let $\sfl(n, m)$ be the suffix length of the set $\sac(n, m)$ with respect
to the lexicographic ordering. This corresponds to the 
number of adjacency boxes in $\sac(n, m)$ as shown in  \figureRef{fig-box-diagram}. In this 
diagram we list ascending compositions in lexicographic order, and draw a box around 
parts in adjacent partitions that are equal. Since parts in adjacent compositions 
are equal, we need one `write' operation for that part when generating the set.
We count these boxes by 
first noting that there is exactly one box in the set $\sac(n, m)$ when $m > \floor{ n /2}$.
Then, for each possible value for the smallest part $x$, there is one box, plus
the boxes in $\sac(n - x, x)$. We therefore get the following recurrence
\begin{equation}\label{eqn-sfl-iter}
\sfl(n, m) = 1 + \sum_{x = m}^{\floor{n / 2}}\left( \sfl(n - x, x) + 1\right),
\end{equation}
which holds for all $1\leq m\leq n$.  

\begin{theorem}\label{the-sfl}
If $n \geq 1$, then $\sfl(n, m) = 2\nac(n, m) - 1$ for all $1 \leq m \leq n$.
\end{theorem}
\begin{proof}
Proceed by strong induction on $n$.  
For $n = 1$, we immediately have  $\sfl(1, 1) = 2\nac(1,1) - 1$ from \equationRef{eqn-nac-iter}
and \equationRef{eqn-sfl-iter}.

Suppose, then, that $\sfl(n', m) = 2\nac(n', m) - 1$ when 
$1 \leq m \leq n'$, for all $1 < n' < n$, and some integer $n$.
Since $n' - x < n' $ for all $1 \leq m \leq x \leq \floor{n'/2} < n'$, we get 
\[
\sfl(n', m)  = 1 + 2\sum_{x = m}^{\floor{n' / 2}}\nac(n' - x, x)
\]
by \equationRef{eqn-sfl-iter} and  the inductive hypothesis. Rearranging, we get
\[
\sfl(n', m)  = 2\left(1 + \sum_{x = m}^{\floor{n' / 2}}\nac(n' - x, x)\right) - 1.
\]
Substituting for \equationRef{eqn-nac-iter} we get $\sfl(n', m)  = 2\nac(n', m) - 1$,
as required.
\jkqed\end{proof}

\section{Lexicographic Successor}\label{sec-lex-succ}
Let $\lexminA(n, m)$ be the lexicographically least element of the set $\sac(n, m)$,
and define $\singletonSequence{n}$ to be the singleton composition $a_1 = n$. It 
is easy to see that 
\begin{equation}\label{eqn-lexminA} 
\lexminA(n, m) = \overbrace{m\dots m}^{\mu}\concatenate
\singletonSequence{n - \mu m},
\end{equation}
where $\mu = \floor{ n / m } - 1$, since the prefix of $\lexminA(n, m)$ 
must be a sequence of $m$ of maximum length.
Then, if we let $\lexsuccA(\sequence{a}{k})$ be the lexicographic successor of 
$\sequence{a}{k}$ for any $\sequence{a}{k} \in \sac(n) \setminus 
\{\singletonSequence{n} \}$, we have 
\begin{equation}\label{eqn-lexsuccA}
\lexsuccA(\sequence{a}{k}) =  \sequence{a}{k - 2} \concatenate
\overbrace{m\dots m}^{\mu}
\concatenate\singletonSequence{n' - \mu m}
\end{equation}
where $m = a_{k - 1} + 1$, $n' = a_{k - 1} + a_{k}$, and $\mu = \floor{n'
/ m} - 1$. We can see that \equationRef{eqn-lexsuccA} correctly generates the 
lexicographic successor of $\sequence{a}{k}$ because it appends the lexicographically
least element of the set $\sac(a_{k - 1} + a_k, a_{k - 1} + 1)$ to $\sequence{a}{k - 2}$
(see Kelleher~\cite[\chapterNumber{5}]{kelleher-encoding} for a full proof).

\begin{theorem}\label{the-ak-nac}
For all $n \geq 1$
\begin{equation}\label{eqn-ak-nac}
2\nac(n, m)  - 1=  \left\lfloor\frac{n(1 - m)}{m}\right\rfloor + 
\stacksum{\sequence{a}{k} \in}{\sac(n)}
{\left\lfloor \frac{a_{k - 1} + a_k}{a_{k - 1} + 1}
\right\rfloor}
\end{equation}
\end{theorem}
\begin{proof}
The number of write operations required to generate the lexicographically least 
element of $\sac(n, m)$ is $\floor{n/m}$ by \equationRef{eqn-lexminA},
since there are $\floor{n/m} - 1$ copies of $m$ written, followed by  one write to insert
the remainder at the end of the composition.  The remaining 
write operations are then counted by summing the number of writes required by 
\equationRef{eqn-lexsuccA} over all elements of $\sac(n) \setminus 
\{\singletonSequence{n} \}$. The number of writes required to transition 
from an arbitrary composition $\sequence{a}{k}$ to its lexicographic 
successor is $\floor{(a_k + a_{k - 1}) / (a_{k - 1} + 1)}$ 
by~\equationRef{eqn-lexsuccA}.  Therefore, we get 
\[
\sfl(n, m) = \floor{n / m} + \stacksum{\sequence{a}{k} \in}{\sac(n) \setminus\{\singletonSequence{n} \} }
{\left\lfloor \frac{a_{k - 1} + a_k}{a_{k - 1} + 1}
\right\rfloor}
\]
Then, since $\floor{(a_{k - 1} + a_k) / (a_{k - 1} + 1)} = n$ when $a_1 = n$ (and $a_0 = 0$), we get 
\[
\sfl(n, m) = n - \floor{n / m} + \stacksum{\sequence{a}{k} \in}{\sac(n) }
{\left\lfloor \frac{a_{k - 1} + a_k}{a_{k - 1} + 1}
\right\rfloor}
\]
Substituting for $\sfl(n, m) = 2\nac(n, m) - 1$ using~\theoremRef{the-sfl} we get 
\equationRef{eqn-ak-nac}, as required.
\jkqed\end{proof}

Setting $m = 1$ in \equationRef{eqn-ak-nac} we get
\equationRef{eqn-ak-pn}, and therefore establish the main result.


\bibliographystyle{plain}    
\bibliography{tlpffp}  

\end{document}